\documentclass[11pt,a4paper]{article}
\usepackage[lmargin=1.0in,rmargin=1.0in,bottom=1.0in,top=1.0in,twoside=False]{geometry}

\usepackage{authblk}

\usepackage{amssymb,amsmath}
\usepackage{graphicx}
\usepackage{enumerate}
\usepackage{tikz}
\usetikzlibrary{shapes}

\usepackage[T1]{fontenc}

 \usepackage{xcolor}
 \usepackage{mathtools}
\usepackage{microtype}
\usepackage{amsfonts}
\usepackage{comment}
\usepackage{mathrsfs}
\usepackage[vlined, ruled, linesnumbered]{algorithm2e}
\DontPrintSemicolon

\usepackage{diagbox}

\usepackage{trimspaces}
\usepackage{nccfoots}
\usepackage{setspace}
\usepackage{inconsolata}
\usepackage{libertine}
\usepackage[absolute]{textpos}
\usepackage{thmtools}
\usepackage{thm-restate}

\usepackage{todonotes}

\usepackage{longtable}

\definecolor{blue}{rgb}{0.1,0.2,0.5}
\definecolor{brown}{rgb}{0.6,0.6,0.2}
\usepackage[ocgcolorlinks, linkcolor={blue}, citecolor={brown}]{hyperref}

\usepackage[amsmath,thmmarks,hyperref]{ntheorem}
\usepackage{cleveref}

\crefformat{page}{#2page~#1#3}%
\Crefformat{page}{#2Page~#1#3}%
\crefformat{equation}{#2(#1)#3}%
\Crefformat{equation}{#2(#1)#3}%
\crefformat{figure}{#2Figure~#1#3}%
\Crefformat{figure}{#2Figure~#1#3}%
\crefformat{section}{#2Section~#1#3}
\Crefformat{section}{#2Section~#1#3}
\crefformat{chapter}{#2Chapter~#1#3}
\Crefformat{chapter}{#2Chapter~#1#3}
\crefformat{chapter*}{#2Chapter~#1#3}
\Crefformat{chapter*}{#2Chapter~#1#3}
\crefformat{part}{#2Part~#1#3}
\Crefformat{part}{#2Part~#1#3}
\crefformat{enumi}{#2(#1)#3}
\Crefformat{enumi}{#2(#1)#3}

\usepackage{enumerate}

\usepackage{latexsym}

\theoremnumbering{arabic}
\theoremstyle{plain}
\theoremsymbol{}
\theorembodyfont{\itshape}
\theoremheaderfont{\normalfont\bfseries}
\theoremseparator{.}

\newtheorem{theorem}{Theorem}
\crefformat{theorem}{#2Theorem~#1#3}
\Crefformat{theorem}{#2Theorem~#1#3}

\newcommand{\newtheoremwithcrefformat}[2]{%
  \newtheorem{#1}[theorem]{#2}%
  \crefformat{#1}{##2\MakeUppercase#1~##1##3}%
  \Crefformat{#1}{##2\MakeUppercase#1~##1##3}%
}
\newcommand{\newseptheoremwithcrefformat}[2]{%
  \newtheorem{#1}{#2}%
  \crefformat{#1}{##2\MakeUppercase#1~##1##3}%
  \Crefformat{#1}{##2\MakeUppercase#1~##1##3}%
}

\newtheoremwithcrefformat{proposition}{Proposition}
\newtheoremwithcrefformat{observation}{Observation}
\newtheoremwithcrefformat{conjecture}{Conjecture}
\newtheoremwithcrefformat{corollary}{Corollary}
\newseptheoremwithcrefformat{claim}{Claim}
\theorembodyfont{\upshape}
\newtheoremwithcrefformat{example}{Example}
\newtheoremwithcrefformat{remark}{Remark}
\newseptheoremwithcrefformat{definition}{Definition}
\newseptheoremwithcrefformat{question}{Question}

\crefname{theorem}{Theorem}{Theorems}

\theoremstyle{nonumberplain}
\theoremheaderfont{\scshape}
\theorembodyfont{\normalfont}
\theoremsymbol{\ensuremath{\square}}
\newtheorem{proof}{Proof}

\theoremsymbol{\ensuremath{\lrcorner}}

\def\cqedsymbol{\ifmmode$\lrcorner$\else{\unskip\nobreak\hfil
\penalty50\hskip1em\null\nobreak\hfil$\lrcorner$
\parfillskip=0pt\finalhyphendemerits=0\endgraf}\fi}

\tikzset{
    position/.style args={#1:#2 from #3}{
        at=(#3.#1), anchor=#1+180, shift=(#1:#2)
    }
}

\let\originalleft\left
\let\originalright\right
\renewcommand{\left}{\mathopen{}\mathclose\bgroup\originalleft}
\renewcommand{\right}{\aftergroup\egroup\originalright}

\usepackage{scalefnt}
\usepackage{caption}
\usepackage{xspace}
\usepackage{amsmath}
\usepackage{graphicx,url}
\usepackage{booktabs}
\usepackage{float}
\usepackage{subfigure}
\usepackage[ruled,linesnumbered,vlined]{algorithm2e}
\usepackage{algpseudocode}
\usepackage{amssymb}
\usepackage{hf-tikz}
\usepackage{marvosym}
\usepackage{textgreek}

\usepackage{diagbox}

\definecolor{RED}{RGB}{255,0,0}

\usepackage{amssymb,boxedminipage,amsmath}
\usepackage[normalem]{ulem}
\usepackage[export]{adjustbox}

\usepackage{enumitem}

\usepackage{tikz}
 \usetikzlibrary{calc}
 \usepackage{boxedminipage}
\usepackage{framed}
\usepackage{xspace}
\usepackage{xifthen}
 \usepackage{tabularx}

\newlength{\RoundedBoxWidth}
\newsavebox{\GrayRoundedBox}
\newenvironment{GrayBox}[1]%
   {\setlength{\RoundedBoxWidth}{.93\textwidth}
    \def\boxheading{#1}
    \begin{lrbox}{\GrayRoundedBox}
       \begin{minipage}{\RoundedBoxWidth}}%
   {   \end{minipage}
    \end{lrbox}
    \begin{center}
    \begin{tikzpicture}%
       \node(Text)[draw=black!20,fill=white,rounded corners,%
             inner sep=2ex,text width=\RoundedBoxWidth]%
             {\usebox{\GrayRoundedBox}};
        \coordinate(x) at (current bounding box.north west);
        \node [draw=white,rectangle,inner sep=3pt,anchor=north west,fill=white]
        at ($(x)+(6pt,.75em)$) {\boxheading};
    \end{tikzpicture}
    \end{center}}

\newenvironment{defproblemx}[2][]{\noindent\ignorespaces%
                                \FrameSep=6pt%
                                \parindent=0pt%
                \vspace*{-1.5em}
                \ifthenelse{\isempty{#1}}{%
                  \begin{GrayBox}{\textsc{#2}}%
                }{%
                  \begin{GrayBox}{\textsc{#2} parameterized by~{#1}}%
                }
                \begin{tabular*}{\textwidth}{@{\hspace{.1em}} >{\itshape} p{1.8cm} p{0.8\textwidth} @{}}%
            }{
                \end{tabular*}%
                \end{GrayBox}%
                \ignorespacesafterend
            }

\newcommand{\defproblema}[3]{%
  \begin{defproblemx}{#1}
    {\bf Instance:}  & #2 \\
    {\bf Goal:} & #3
  \end{defproblemx}
}%

\usepackage{lineno}

\newcommand{\frontpageformat}{arxiv}

\begin{document}

\ifthenelse{\equal{\frontpageformat}{submission}}{%
\author{anonymous}
\title{A note on independent sets in~sparse-dense~graphs}
\begin{titlepage}
\def\thepage{}
\thispagestyle{empty}
\maketitle
}{%
\author[1]{Akanksha Agrawal}
\author[2]{Henning Fernau}
\author[2]{Philipp Kindermann}
\author[2]{Kevin Mann}
\author[3,4]{Uéverton~S.~Souza}

\affil[1]{Indian Institute of Technology Madras,  India}
\affil[2]{University of Trier, Germany}
\affil[3]{Fluminense Federal University, Brazil}
\affil[4]{University of Warsaw, Poland}

\title{Recognizing well-dominated graphs is coNP-complete%
\thanks{
This research has received funding from Rio de Janeiro Research Support Foundation (FAPERJ) under grant agreement E-26/201.344/2021,  National Council for Scientific and Technological Development (CNPq) under grant agreement 309832/2020-9, and the European Research Council (ERC) under the European Union's Horizon 2020 research and innovation programme under grant agreement CUTACOMBS (No. 714704).\\
{akanksha@cse.iitm.ac.in},  {fernau@uni-trier.de}, {kindermann@uni-trier.de}, {mann@uni-trier.de},  {ueverton@ic.uff.br}.
}}

\begin{titlepage}
\def\thepage{}
\thispagestyle{empty}
\maketitle
\begin{textblock}{20}(0, 13.3)
\includegraphics[width=40px]{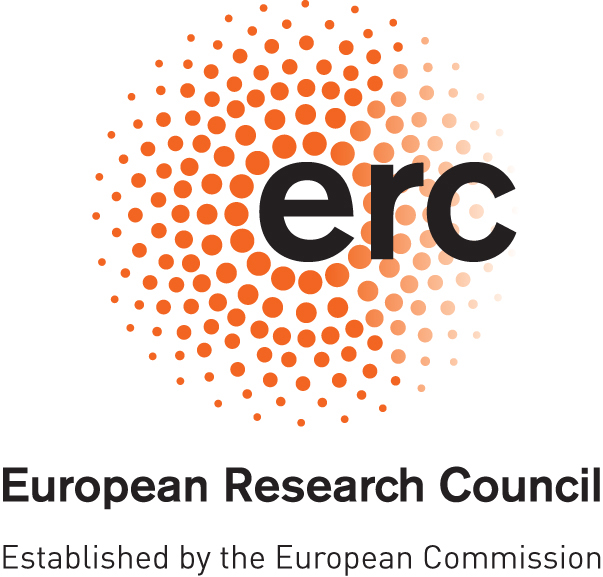}%
\end{textblock}
\begin{textblock}{20}(0, 14.1)
\includegraphics[width=40px]{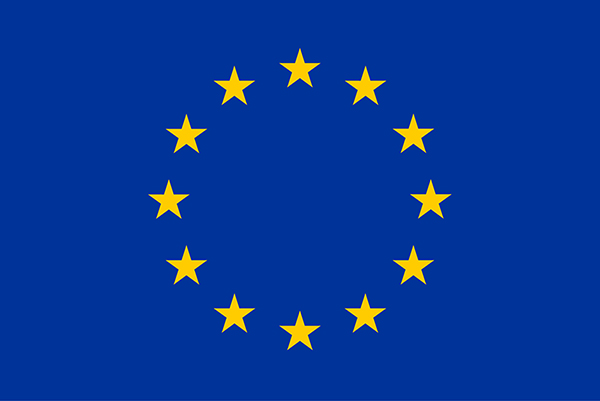}%
\end{textblock}
}

\begin{abstract}
A graph~$G$ is well-covered if every minimal vertex cover of~$G$ is minimum,  and a graph~$G$ is well-dominated if every minimal dominating set of~$G$ is minimum.  
Studies on well-covered graphs were initiated in [Plummer,  JCT~1970], and well-dominated graphs were first introduced in  [Finbow,  Hartnell and Nowakow,  AC~1988]. Well-dominated graphs are well-covered, and both classes have been widely studied in the literature. The recognition of well-covered graphs was proved {\sf coNP}-complete by [Chv\'atal and Slater, AODM 1993] and by [Sankaranarayana and Stewart, Networks  1992], but the complexity of recognizing well-domi\-nated graphs has been left open since their introduction. We close this complexity gap by proving that recognizing well-dominated graphs is {\sf coNP}-complete. This solves a well-known open question (c.f. [Levit and Tankus,  DM 2017] and [G{\"{o}}z{\"{u}}pek,  Hujdurovic and Milani{\v c}, DMTCS  2017]), which was first asked in [Caro, Sebő and Tarsi, JAlg 1996].
Surprisingly, our proof is quite simple, although it was a long-standing open problem. 
Finally, we show that recognizing well-totally-domi\-nated graphs is {\sf coNP}-complete,  answering a question of [Bahad{\i}r,  Ekim,  and G{\"o}z{\"u}pek, AMC  2021].  
\medskip

\end{abstract}

\end{titlepage}

\section{Introduction}
\label{s_intro}

{\sc Minimum Dominating Set},  {\sc Maximum Independent Set},  and {\sc Minimum Vertex Cover} are some of the most important computational and combinatorial problems, having a number of ``real world'' relevant applications and appearing in a wide range of natural situations.  These problems cannot be solved in polynomial time unless $\textsf{P}=\textsf{NP}$, since they were proved to be {\sf NP}-hard in 1972 in the seminal paper of Karp~\cite{Karp72}.
In spite of this fact,  a {\em minimal} dominating set,  and a {\sl maximal} independent set of a graph can be found in polynomial time using a greedy algorithm.  Also, the complement of a maximal independent set is a minimal vertex cover, so the same applies for vertex covers.

In 1970,  Plummer~\cite{PLUMMER197091} defined \emph{well-covered} graphs as the class of graphs $G$ where every minimal vertex cover is also a minimum vertex cover. This is equivalent to requesting that all maximal independent sets have the same cardinality.  Therefore,  well-covered graphs form a natural graph class for which {\sc Maximum Independent Set}  and {\sc Minimum Vertex Cover} can be solved in polynomial time. 

In the 1990's, the problem of recognizing a graph in the class of well-covered graphs,  called \textsc{Well-Coveredness}, was independently proved to be {\sf {coNP}}-complete  by Chv\'atal and Slater~\cite{chvatal1993note} and  by Sankaranarayana and Stewart~\cite{SaSt92}.  
In addition, structural characterizations or polynomial-time algorithms for recognizing well-covered graphs were studied on claw-free graphs~\cite{lesk1984equimatchable,TANKUS1996293}, graphs without large cycles~\cite{doi:10.1002/jgt.3190180707},
block-cactus graphs~\cite{randerath1994characterization},
bipartite graphs~\cite{ravindra1977well},
graphs with large girth~\cite{FINBOW199344},  
$P_4$-sparse graphs~\cite{DBLP:journals/gc/KleinMM13},
planar, chordal,  and circular arc graphs~\cite{doi:10.1002/(SICI)1097-0118(199602)21:2<113::AID-JGT1>3.0.CO;2-U}, bounded degree graphs~\cite{fcf48a9c6bc04e0f8917664daa376300},
and perfect graphs of bounded clique size~\cite{DeanZ94}.
A survey on well-covered graphs due to Plummer from 1993 can be found in~\cite{surveyPlummer}.  
In addition,  Sankaranarayana and Stewart determined the complexity of several problems on well-covered graphs~\cite{SaSt92}. In 2011,  Brown and Hoshino~\cite{BROWN2011244} showed that recognizing well-covered graphs is {\sf {coNP}}-complete even when restricted to the family of circulant graphs.  In 2020, Alves, Couto, Faria, Gravier, Klein, and Souza~\cite{AlvesFKGSS20} studied the complexity of the {\sc Graph Sandwich Problem} for the property of being a well-covered graph whose vertex set can be partitioned into~$k$ independent sets and into~$\ell$ cliques for fixed integers~$k$ and~$\ell$, i.e., well-covered graphs that are also $(k,\ell)$-graphs.  In 2021, Faria and Souza~\cite{cocoon21} studied the complexity of {\sc Probe Problem} for the property of being a $(k,\ell)$-graph that is well-covered.  Also, a polynomial-time algorithm for recognizing some sparse-dense graphs that are well-covered can be found in~\cite{https://doi.org/10.48550/arxiv.2208.04408}.

Regarding parameterized complexity, in 2018,  Alves,  Dabrowski,  Faria,  Klein, Sau, and Souza proved that determining whether every minimal vertex cover of a given graph~$G$ has size~$k$ is fixed-parameter tractable with respect to~$k$, but the problem of determining whether every maximal independent set of~$G$ has size~$k$ is  {\sf {coW[2]}}-hard, considering~$k$ as parameter. 
This last result illustrates that when considering the ``wellness'' variant of graph problems~$\Pi$, there can be a leap in terms of complexity when analyzing the resulting well version of~$\Pi$.  Recall that {\sc $k$-Independent Set} is one of the canonical problems of the class {\sf W[1]} but its well version is hard for {\sf coW[2]}.  In 2019, Araújo,  Costa, Klein,  Sampaio, and Souza showed that the problem of determining whether every minimal vertex cover of a graph $G$ has size~$k$ admits a kernel having $\mathcal{O}(k)$ vertices.  In addition, parameterized algorithms for \textsc{Well-Coveredness} considering other structural parameterizations can also be found in~\cite{AlvesDFKSS18,DBLP:journals/dmtcs/AraujoCKSS19}. 

Since every maximal independent set is a minimal dominating set,  the well-covered graph class is a superclass of the class of graphs whose all minimal dominating sets have the same size.  Such graphs were first studied in 1988 by Finbow,  Hartnell, and Nowakow~\cite{FHN88} and are called \emph{well-dominated} graphs.  
The structure of well-dominated bipartite graphs and well-dominated graphs with no cycle of length less than~5 were analyzed in~\cite{FHN88}.  Also,  structural characterizations of well-dominated block graphs and unicyclic graphs were presented in~\cite{volkmann1990well}, and well-dominated chordal graphs were characterized in~\cite{doi:10.1002/(SICI)1097-0118(199602)21:2<113::AID-JGT1>3.0.CO;2-U,Wellirredundant}.
In 2011,  a characterization of 4-connected, 4-regular, claw-free, well-dominated graphs was given, see~\cite{GIONET20111225}.  
In 2017,  Levit and Tankus proved that every well-covered graph without cycles of lengths~4 and~5 is well-dominated~\cite{LEVIT20171793}.  In the same year,  G{\"{o}}z{\"{u}}pek,  Hujdurovic,  Milani{\v c}~\cite{dmtcs:milanic} presented a characterization of well-dominated graphs with domination number two and show that well-dominated graphs
can be recognized in polynomial time in any class of graphs with bounded domination number.
In 2021,  Anderson,  Kuenzel, and Rall~\cite{anderson2021well} showed that there are exactly eleven connected, well-dominated, triangle-free graphs whose domination number is at most~3, and Rall~\cite{https://doi.org/10.48550/arxiv.2105.09797} gave a complete characterization of nontrivial direct products that are well-dominated.  

Well-covered graphs~$G$ with no isolated vertices such that every maximal independent set has size $\frac{|V(G)|}{2}$ form the class of \emph{very well-covered graphs}.  Similarly,  one can define the \emph{very well-dominated graphs}.  Very well-covered graphs and very well-dominated graphs can be recognized in polynomial-time due to their structural characterizations, see~\cite{FAVARON1982177,Wellirredundant}.

Although well-dominated graphs have been widely studied from 1988 to nowadays, the complexity status of recognizing well-dominated graphs was unknown until this current work.
In this paper, we show that recognizing well-dominated graphs is {\sf coNP}-complete,  solving a well-known open question (c.f.~\cite{CARO1996137,dmtcs:milanic,LEVIT20171793}). To the best of our knowledge, the first time the recognition of well-dominated graphs was explicitly stated as an open question was in 1996 by Caro, Sebő and Tarsi~\cite{CARO1996137}. In~\cite{CARO1996137}, well-dominated graphs are called greedy instances of the {\sc Minimum Dominating Set} problem.

Besides that, analogously to well-dominated graphs, in 1997,  Hartnell and Rall~\cite{rall1997graphs} initiated the study of graphs all whose minimal total dominating sets are of the same size.  Such graphs are called 
\emph{well-totally-dominated} graphs.  In 2021,  Bahad{\i}r,  Ekim,  and G{\"o}z{\"u}pek~\cite{AMC2465} showed among other results that well-totally-dominated graphs having bounded total domination number can be recognized in polynomial time. They left as an open question the complexity of recognizing well-totally-dominated graphs.  In this paper,  we also  answer this question, showing that the recognition of well-totally-dominated graphs is {\sf coNP}-complete.

We consider the following ``wellness'' problems related to domination and covering.


\defproblema{Well-Coveredness} 
{A graph $G=(V,E)$.}
{Determine whether every minimal vertex cover of $G$ has the same size.

\emph{Note:} A vertex cover of $G$ is a subset of $V(G)$ intersecting all edges of $G$.}

\defproblema{Well-Domination} 
{A graph $G=(V,E)$.}
{Determine whether every minimal dominating set of $G$ has the same size.

\emph{Note:} A dominating set of $G$ is a subset $S\subseteq V(G)$ such that each vertex $v\in V(G)\setminus S$ has a neighbor in $S$.}

\defproblema{Well-Total Domination} 
{A graph $G=(V,E)$.}
{Determine whether whether every minimal total dominating set of~$G$ has the same size.

\emph{Note:} A total dominating set of~$G$ is a subset $S\subseteq V(G)$ such that any vertex of $G$ has a neighbor in~$S$, including vertices of~$S$.}

\defproblema{Well-Hitting Set} 
{An universe set $\mathcal{U}$ of elements, and a family $\mathcal{F}$ of subsets of elements of $\mathcal{U}$.}
{Determine whether every minimal hitting set of $(\mathcal{U},\mathcal{F})$ has the same size. 

\emph{Note:} a hitting set of $(\mathcal{U},\mathcal{F})$ is a subset of $\mathcal{U}$ intersecting every set in $\mathcal{F}$.}

\defproblema{Well-Set Cover} 
{An universe set $\mathcal{U}$ of elements, and a family $\mathcal{F}$ of subsets of elements of $\mathcal{U}$.}
{Determine whether every minimal set cover of $(\mathcal{U},\mathcal{F})$ has the same size. 

Note: a set cover of $(\mathcal{U},\mathcal{F})$ is a subset $\mathcal{S}$ of $\mathcal{F}$ such that every element of $\mathcal{U}$ is contained in at least one set in $\mathcal{S}$.}

\defproblema{Well-Hitting-Set Cover} 
{An universe set $\mathcal{U}$ of elements, and a family $\mathcal{F}$ of subsets of elements of $\mathcal{U}$.}
{Determine whether every minimal set cover and every minimal hitting set of $(\mathcal{U},\mathcal{F})$ has the same size. }

In this paper, we see that all these problems are {\sf coNP}-complete.


\section{Computational Complexity}

First, observe that all problems studied in this work are in {\sf coNP}, since any pair of minimal solutions having different sizes certifies {\em no}-instances for them.  In addition, the minimality of solutions for such problems can easily checked in polynomial time. A similar general observation is contained in~\cite{CARO1996137}. Thus, we focus on {\sf coNP}-hardness in this section.

\medskip

Any \textsc{Vertex Cover} instance can be interpreted as a \textsc{Hitting Set} instance or a \textsc{Set Cover} instance. Therefore,  the following corollary holds as a consequence of the {\sf coNP}-completeness of {\sc Well-Coveredness}~\cite{chvatal1993note,SaSt92}.  

\begin{corollary}\label{cor:WellHittingSet}
{\sc Well-Hitting Set} and {\sc Well-Set Cover} are {\sf coNP}-complete.
\end{corollary}

Recall that, for $n>1$,  any $n$-vertex connected bipartite graph~$G$ is well-dominated if and only if every minimal dominating set of~$G$ has size $\frac{n}{2}$, because any maximal independent set is also a minimal dominating set.  Thus,  bipartite well-dominated  graphs are very well-dominated and so recognized in polynomial time, since they are either a $C_4$ or the corona product of a connected graph with a $K_1$,  see~\cite{Wellirredundant}. Contrastingly, the next result shows that {\sc Well-Total Domination} on bipartite graphs is unlikely to be polynomial-time solvable.

\begin{theorem}
{\sc Well-Total Domination} on bipartite graphs is {\sf coNP}-complete.
\end{theorem}

\begin{proof}
Let $H=(\mathcal{U},\mathcal{F})$ be a hypergraph ($\mathcal{F}\subseteq 2^\mathcal{U}$) such that each hitting set has at least two element. Then define $G=(V,E)$ with 
\begin{eqnarray*}
        V&=& \{s,t\} \cup \{v_u\mid u\in \mathcal{U}\} \cup \{ w_F \mid F\in \mathcal{F}\},\\
        E&=& \{\{s,t\}\} \cup \{\{s,v_u\}\mid u\in U\}\cup \{ \{ v_u,w_F\} \mid u\in F \mbox{ and } F \in \mathcal{F} \}.
\end{eqnarray*}

Since $t$ has only the neighbor~$s$,  vertex~$s$ has to be in each total dominating set.

Let $Z$ be a minimal hitting set of~$H$.  Define $D= \{s\}\cup \{v_z\mid z\in Z\}$. 
We want to show that $D$ is a minimal total dominating set.  As $s\in D$, the vertices in $\{t\}\cup \{v_u \mid u\in \mathcal{U}\}$ are dominated.  Since~$Z$ is not empty, vertex~$s$ is also dominated.  Let $F\in \mathcal{F}$. Then there exists some $u\in Z\cap F$. This implies there exists a  vertex $v_u\in D\cap N(w_F)$.  Therefore, $D$ is a total dominating set of~$G$.  As mentioned before,  $s$ has to be in~$D$ for total domination reasons.  Assume there exists a $v_u\in D$ such that $D\setminus \{v_u\}$ is a total dominating set of $G$. This implies that for each $w_F\in N(v_u)$, there exists a $v_{y_F}\in (N(w_F)\cap D)\setminus \{v_u \}$.  Hence, for each $F\in \mathcal{F}$ with $u\in F$, there exists a $y_F\in F\cap Z$ such that  $y_F\neq u$. This contradicts the minimality of~$Z$.

Let $D$ be a minimal total dominating set of~$G$.  As mentioned before, $s\in D$.  
For each $F\in \mathcal{F}$, $N(w_F)\subseteq \{v_u\mid u\in \mathcal{U}\} \subseteq N(s)$ holds. 
Therefore, $w_F$ is in no minimal total dominating set. 
But since $w_F$ is not dominated by~$s$, there has to be at least one element in $D\cap \{v_u\mid u\in \mathcal{U}\}$,  implying that~$t$ cannot be in a minimal total dominating set.  
Therefore, each minimal total dominating set of~$G$ is a subset of $\{s\}\cup  \{v_u\mid u\in \mathcal{U}\}$. Define $Z=\{u \mid v_u \in D\}$. As~$D$ is a total dominating set, for each $F\in \mathcal{F}$, $w_F$ is dominated. Thus, $Z$ is a hitting set of~$H$.  
Assume that~$Z$ is not minimal.  
This implies the existence of a $u\in Z$ such that for each $F\in \mathcal{F}$ having $u\in F$, there exists a $z_F\in (F \cap Z)\setminus\{u\}$. This implies that for each $w_F\in N(v_u)$, there exists a $v_{z_F}\in (N(w_F)\cap D)\setminus\{v_u\}$. This contradicts the minimality of~$D$.

Therefore, for each minimal hitting set~$Z$ of~$H$, there exists a minimal total dominating set $D_Z$ of~$G$ with $\vert D_Z\vert = \vert Z\vert + 1 $. Conversely, for each minimal total dominating set~$D$ of~$G$, there exists a  minimal hitting set~$Z_D$ of~$H$ with $\vert Z_D\vert = \vert D\vert - 1 $.  
Thus, $H=(\mathcal{U},\mathcal{F})$ is a {\em yes}-instance of {\sc Well-Hitting Set}  if and only if $G$ is a {\em yes}-instance of {\sc Well-Total Domination}.  Therefore,  by Corollary~\ref{cor:WellHittingSet} the claim holds. \qed
\end{proof}

The argument would also hold if $H$ is a {\sc Vertex Cover} instance.  In this case, all vertices in $\{w_F \mid F\in \mathcal{F}\}$ would have degree 2 and $G[V\setminus \{ w_F \mid F\in \mathcal{F}\}]$ is a tree. Therefore, $G$ would be a 2-degenerate graph.  
If we would define $\{s\} \cup \{v_u\mid u\in \mathcal{U}\}$ as a clique, it would not change the argument, either. This modification turns~$G$ into a split graph.  Hence, the following holds.
 
\begin{corollary}\label{cor:totaldom_split}
 {\sc Well-Total Domination} is {\sf coNP}-complete on split or {2-degenerate} bipartite graphs.
\end{corollary}

This result on split graphs is interesting insofar, as  {\sc Well-Domination} on chordal graphs is solvable in polynomial time,  see~\cite{Wellirredundant}. This indicates that {\sc Well-Total Domination} tends to be a more difficult problem to solve than {\sc Well-Domination}.  
However,  next result shows that on general graphs the {\sc Well-Domination} problem is also {\sf coNP}-complete.

\begin{theorem}\label{thm:NPcomplete}
{\sc Well-Domination} is {\sf coNP}-complete.
\end{theorem}
\begin{proof}

Let $\mathcal{(U,F)}$ be an instance of {\sc Well-Hitting Set}, where $\mathcal{U}$ is the universe set and $\mathcal{F}$ is a family of subsets of $\mathcal{U}$.  Let $k$ be the size of a minimal hitting set of $\mathcal{(U,F)}$. 

Note that $\mathcal{(U,F)}$ is a {\em yes}-instance of a {\sc Well-Hitting Set} if and only if every minimal hitting set of $\mathcal{(U,F)}$ has size~$k$.  Since a minimal hitting set of $\mathcal{(U,F)}$ can be obtained in polynomial time, without loss of generality, we assume that~$k$ is given together with $\mathcal{(U,F)}$ and we are asked if every minimal hitting set of $\mathcal{(U,F)}$ has size~$k$.

From $\mathcal{(U,F)}$ we construct a graph~$G$ as follows:
\begin{enumerate}
    \item Define $$V(G)=\{r\}\cup U\cup F_1, F_2\cup \ldots\cup F_{k-1},$$ where~$U$ is a vertex set of size $|\mathcal{U}|$ such that each $v_u\in U$ represents a distinct element~$u$ of the universe~$\mathcal{U}$, and each $F_i$ is a vertex set of size $|\mathcal{F}|$ where each vertex $v_j^i$ of $F_i$ represents a set $S_j\in \mathcal{F}$.  

    \item Define $U$ and each vertex set $F_i$ as cliques of~$G$.
    
    \item Add an edge between a vertex $v_j^i \in F_i$ and a vertex $v_u\in U$ if $u\in S_j$. Note that the subgraph $G[U\cup F_i]$ is isomorphic to the bipartite incidence graph of $\mathcal{(U,F)}$.
    \item Add an edge between $r$ and each vertex in $U$.
\end{enumerate}
This completes the construction of $G$.

Now, we argue that every minimal hitting set of $\mathcal{(U,F)}$ has size~$k$ if and only if $G$ is well dominated.

Notice that~$G$ has a clique cover~$C$ of size~$k$ formed by the cliques $(U\cup\{r\}),F_1,\ldots,F_{k-2}$, and $F_{k-1}$. Also, any minimal dominating set of~$G$ that contains at least one vertex per clique of~$C$, by minimality, must contain exactly one vertex per clique (hence, it has size~$k$). By taking $r$ and one vertex per clique~$F_i$, we know that~$G$ has a minimal dominating set of size~$k$. Besides, by construction, any minimal dominating set of~$G$ contains either $r$ or some vertex of~$U$. Therefore, any minimal dominating set of~$G$ having size different from~$k$ must contain no vertex of $F_1\cup \ldots \cup F_{k-1} \cup \{r\}$, but contain a subset of vertices of~$U$ that dominate each $F_i$. Since $G[U\cup F_i]$ is isomorphic to the bipartite incidence graph of $\mathcal{(U,F)}$, it holds that~$G$ has a minimal dominating set of~$G$ having size different from~$k$ ($G$ is not well-dominated) if and only if $\mathcal{(U,F)}$ has a minimal hitting set of size different than $k$ (so, $\mathcal{(U,F)}$ is not a {\em yes}-instance of {\sc Well-Hitting Set}). \qed
\end{proof}

Finally,  recall that \textsc{Dominating Set} and \textsc{Total Dominating Set} instances~$G$ can be interpreted as \textsc{Hitting Set} or \textsc{Set Cover} instances where the elements of $\mathcal{U}$ are the vertices of $G$ and the family $\mathcal{F}$ is formed by the close neighborhood or the open neighborhood of the vertices of $G$, respectively. 
Also,  from the incidence bipartite graph of the resulting instance $(\mathcal{U},\mathcal{F})$, it is easy to see that the existence of a hitting set of size~$k$ in $(\mathcal{U},\mathcal{F})$ also implies the existence of a set cover having the same size in $(\mathcal{U},\mathcal{F})$,  and vice versa.  Hence,  the following corollary holds as a consequence of the {\sf coNP}-completeness of {\sc Well-Domination} and {\sc Well-Total Domination}.  

\begin{corollary}\label{cor:WellHittingSetCover}
{\sc Well-Hitting-Set Cover} is {\sf coNP}-complete.
\end{corollary}

Corollary~\ref{cor:WellHittingSetCover} completes a {\sf P} versus {\sf coNP}-complete dichotomy regarding well domination problems on bipartite instances. Let $B$ be a bipartite graph having vertex set bipartition $V(B)=V_{\mathcal U}\cup V_{\mathcal F}$.
If the question is if all minimal subsets of $V_{\mathcal U}$ that dominate $V_{\mathcal F}$ have the same size, we are dealing with {\sc Well-Hitting Set}. The converse would be {\sc Well-Set Cover}. Also, one could consider, at the same time, all minimal subsets of $V_{\mathcal U}$ that dominate $V_{\mathcal F}$ and all minimal subsets of $V_{\mathcal F}$ that dominate $V_{\mathcal U}$ having all of them the same size, which is precisely the {\sc Well-Hitting-Set Cover} problem. However, if asked about all minimal subsets of $V_{\mathcal U}\cup V_{\mathcal F}$ that dominate $V_{\mathcal U}\cup V_{\mathcal F}$ having the same size then we are dealing with {\sc Well-Domination} on bipartite graphs, which is polynomial-time solvable, as previously discussed. Therefore, the hardness result of Corollary~\ref{cor:WellHittingSetCover} is tight concerning these constraints.


\section{Conclusions}

We have shown that well variations of a number of combinatorial properties is complete for the complexity class \textsf{coNP}. One algorithmic interpretation of the well variations is that this defines a graph class where a natural greedy strategy always finds the optimum. One could actually relax this requirement and ask, say, in the case of \textsc{Maximum Independent Set}, when the greedy heuristic that always picks a vertex of smallest degree next achieves, say, a factor-2 approximation. (Recall that in general, \textsc{Maximum Independent Set} does not allow polynomial-time constant-factor approximation algorithms.) Alas, determining if a graph can be approximated by a factor of two in this way is again a problem complete for \textsf{coNP}, as shown in \cite{BodThiYam97}. It might be interesting to ask similar questions for  \textsc{Minimum Dominating Set}, for instance, seeing the flavor of results in this paper. 
In 1996, Caro, Sebő and Tarsi~\cite{CARO1996137} had also pointed as a potentially interesting direction the study of instances where a greedy algorithm always guarantees to provide a good approximation of the optimal goal. 

There is also a further combinatorial way of looking at and possibly generalizing the questions discussed in this paper. For instance, one could interpret the \textsc{Well-Domination} problem as asking to decide, for a given graph~$G$, if $\gamma(G)=\Gamma(G)$, i.e., if the lower and upper domination numbers of~$G$ coincide. Likewise, \textsc{Well-Coveredness} aks if $\iota(G)=\alpha(G)$, i.e., if the independent domination number and the independence number of~$G$ coincide. In this spirit, one could ask similar decision questions for other parameters of the famous domination chain, as introduced in~\cite{CocHedMil78}.
For (a survey of) more recent computational results, we refer to 
\cite{DBLP:journals/dam/BazganBCF20,DBLP:journals/tcs/BazganBCFJKLLMP18}. As it is known that all these parameters can differ arbitrarily, the approximation questions discussed in the previous paragraph can be asked analogously for any pair of parameters of the domination chain, or similarly for other graph parameters where inequality relations are known.

\section*{Acknowledgement}

This common project was initiated at the Graphmasters workshop that took place after IWOCA 2022, on June 10th, in Trier, Germany.

Uéverton S. Souza thanks Martin Milanič and András Sebő for presenting him with the problem of recognizing well-dominated graphs on the occasion of the IX Latin and American Algorithms, Graphs and Optimization Symposium (LAGOS 2017).

\bibliographystyle{plainurl}
\bibliography{references}

\end{document}